\pdfoutput=1

\makeatletter
\IfFileExists{./numapde-latex/numapde-packages.sty}{\providecommand*{\input@path}{}\edef\input@path{{./numapde-latex/}\input@path}}{}
\makeatother

\documentclass{numapde-preprint}

\usepackage{numapde-semantic}
\usepackage{numapde-style}
\usepackage{numapde-local}
\usepackage{numapde-syntax}

\IfFileExists{./numapde-bibliography/numapde.bib}{\addbibresource{./numapde-bibliography/numapde.bib}}{\addbibresource{numapde.bib}}

\hypersetup{
	pdftitle={Preconditioning for a Phase-Field Model with Application to Morphology Evolution in Organic Semiconductors},
	pdfauthor={Kai Bergermann, Carsten Deibel, Roland Herzog, Roderick C. I. MacKenzie, Jan-Frederik Pietschmann, Martin Stoll},
	pdfkeywords={preconditioning, phase--field models, organic solar cells, Cahn--Hilliard, finite element analysis}
}

\title{Preconditioning for a Phase-Field Model with Application to Morphology Evolution in Organic Semiconductors}
\subtitle{}
\shorttitle{Preconditioning for Morphology Evolution in Organic Semiconductors}

\author{Kai Bergermann\thanks{Technische Universität Chemnitz, Faculty of Mathematics, 09107 Chemnitz, Germany (\email{kai.bergermann@math.tu-chemnitz.de}, \url{https://www.tu-chemnitz.de/mathematik/wire/people/bergermann.php}, \orcid{0000-0002-2259-1839}, \email{jfpietschmann@math.tu-chemnitz.de}, \url{https://www.tu-chemnitz.de/mathematik/invpde/}, \orcid{0000-0003-0383-8696}, \email{martin.stoll@math.tu-chemnitz.de}, \url{https://www.tu-chemnitz.de/mathematik/wire/prof.php}, \orcid{0000-0003-0951-4756}).}
\and
Carsten Deibel\thanks{Technische Universität Chemnitz, Department of Physics, 09126 Chemnitz, Germany (\email{deibel@physik.tu-chemnitz.de}, \url{https://www.tu-chemnitz.de/physik/OPKM/}, \orcid{0000-0002-3061-7234}).}
\and
Roland Herzog\thanks{Interdisciplinary Center for Scientific Computing, Heidelberg University, 69120 Heidelberg, Germany (\email{roland.herzog@iwr.uni-heidelberg.de}, \url{https://scoop.iwr.uni-heidelberg.de}, \orcid{0000-0003-2164-6575}).}
\and
Roderick C. I. MacKenzie\thanks{Durham University, Department of Engineering, Durham, UK (\email{roderick.mackenzie@durham.ac.uk}, \url{https://www.durham.ac.uk/staff/roderick-mackenzie/}, \orcid{0000-0002-8833-2872}).}
\and
Jan-Frederik Pietschmann\footnotemark[1]
\and
Martin Stoll\footnotemark[1]}
\shortauthor{K. Bergermann, C. Deibel, R. Herzog, R. C. I. MacKenzie, J.-F. Pietschmann and M. Stoll}

\dedication{}

\IfFileExists{numapde-uncolorizeHyperrefIfChanges.sty}{\RequirePackage{numapde-uncolorizeHyperrefIfChanges}}{}

\begin{document}
\maketitle

\begin{abstract}
The Cahn--Hilliard equations are a versatile model for describing the evolution of complex morphologies. 
In this paper we present a computational pipeline for the numerical solution of a ternary phase-field model for describing the nanomorphology of donor--acceptor semiconductor blends used in organic photovoltaic devices. 
The model consists of two coupled fourth-order partial differential equations that are discretized using a finite element approach. 
In order to solve the resulting large-scale linear systems efficiently, we propose a preconditioning strategy that is based on efficient approximations of the Schur-complement of a saddle point system. 
We show that this approach performs robustly with respect to variations in the discretization parameters. 
Finally, we outline that the computed morphologies can be used for the computation of charge generation, recombination, and transport in organic solar cells.\end{abstract}



\section{Introduction}
\label{section:introduction}

We consider a model for solvent-based fabrication of organic solar cells.
A thin-film of a dilute blend containing electron-acceptor, electron-donor and solvent is deposited on a substrate. 
As the solvent evaporates, the initially homogeneous mixture undergoes phase separation into electron-acceptor rich and electron-donor rich areas. 
In order to simulate the evolution of the morphology we use a phase-field model based on the Cahn--Hilliard equation \cite{CahnHilliard:1958:1}, which is a fourth order partial differential equation (PDE). 
The equations are derived from the minimization of the Ginzburg--Landau energy function via a gradient flow. 
The original Cahn--Hilliard equation models the evolution of two phases while in our case we require a system of three components and follow the model introduced in \cite{WodoGanapathysubramanian:2012:1}.
Our focus here is on establishing the model equations and after discretization via finite elements we focus on the preconditioning of the linear systems. 
As these will be of large scale and the system will be ill-conditioned, the convergence of any iterative solver will be slow unless we introduce a suitable preconditioning strategy. 
Our preconditioning approach for the ternary Cahn--Hilliard system is based on applying block-preconditioners \cite{BoschStollBenner:2014:1,BoschStoll:2015:2,BoyanovaDoQuangNeytcheva:2012:1,Zulehner:2011:1} and these rely on the coupling of well studied components such as algebraic multigrid methods.
Finally, in a proof-of-concept study we demonstrate how these morphologies can be incorporated into 2D electrical device simulations of organic solar cells. 
We demonstrate that our generated morphologies can significantly affect both the current voltage curves and the charge density within the active layer.

\section{Phase--Field Model}
\label{section:model}

The mathematical description of the morphology evolution can be done by a phase field model \cite[Chapter~10]{Raabe:1998:1}, that consists of a domain $\Omega \subset \R^3$ and three scalar fields 
\begin{equation}
	\phi_p, \; \phi_\nfa, \; \phi_s \colon \Omega \times [0,T] \rightarrow [0,1] \subset \R
\end{equation}
representing the volume fractions of polymer, non-fullerene acceptor (NFA) and solvent, respectively, at a given point in the domain at a given time in the interval $[0,T]$.
In this work, we focus on the numerical treatment of the model and leave the scaling of physical dimensions to future work.
The conservation relation
\begin{equation}\label{eq:sumphis}
	\phi_p + \phi_\nfa + \phi_s 
	= 
	1
\end{equation}
applies for all $x \in \Omega$ and every $t \in [0,T]$.
The Ginzburg--Landau energy functional forms the basis for the Cahn--Hilliard equation and is given by
\begin{equation}
	\label{eq:ginzburglandau}
	F(\phi_p, \phi_\nfa, \phi_s) 
	=
	\int_{\Omega} \paren[auto][]{f(\phi_p, \phi_\nfa, \phi_s) + \frac{\epsilon_p}{2} \abs[auto]{\nabla \phi_p}^2 + \frac{\epsilon_\nfa}{2} \abs[auto]{ \nabla \phi_\nfa }^2} \d x 
	+ 
	F_s(\phi_p, \phi_\nfa, x)
	,
\end{equation}
where $f(\phi_p, \phi_\nfa, \phi_s)$ accounts for the bulk energy which, according to the Flory--Huggins theory \cite{Flory:1953:1}, is chosen as 
\begin{multline}
	\label{eq:logarithmicpotential}
	f(\phi_p,\phi_\nfa,\phi_s) 
	= 
	\paren[Big][.{\frac{\phi_p}{N_p} \ln(\phi_p) + \frac{\phi_\nfa}{N_\nfa} \ln(\phi_\nfa) + \frac{\phi_s}{N_s} \ln(\phi_s)}
	\\
	\paren[Big].]{{}+ \chi_{p,\nfa} \phi_p \phi_\nfa + \chi_{p,s} \phi_p \phi_s + \chi_{\nfa,s} \phi_\nfa \phi_s}
	.
\end{multline}
The parameters $\epsilon_p$, $\epsilon_\nfa$ represent (ideally small) interface parameters that control the width of the transition layers between the different components.
Their values are chosen as $\epsilon_p = \epsilon_\nfa = 10^{-3}$ in most of our numerical experiments and \Cref{figure:minres_iterations_runtime_varying_epsilon} shows a comparison of different interface parameters.
In the numerical experiment section we replace the logarithmic bulk energy term \eqref{eq:logarithmicpotential} by a polynomial approximation, which is easier to handle numerically.

Furthermore, we denote by $\Gamma_b$ the part of the boundary $\partial\Omega$ that is in contact with the substrate material while $\Gamma_t$ denotes the top part where evaporation occurs.
Interactions with the substrate are incorporated by the quantity $F_s(\phi_p,\phi_\nfa,x)$, which acts only on the surface $\Gamma_b$. 
For our application, we chose a density of the form 
\begin{equation}\label{eq:surfenergy}
	f_s(\phi_p, \phi_\nfa, x) 
	= 
	p_p(x) (g_p \phi_p + h_p \phi_p^2)
	+ 
	p_\nfa(x) \paren[auto](){g_\nfa \phi_\nfa + h_\nfa \phi_\nfa^2}
	,
\end{equation}
where $p_p(x)$ and $p_\nfa(x)$ are functions defined for all points $x \in \Gamma_b$ on the film's interface with the substrate material.
Integration over the surface $\Gamma_b$ and multiplication with the usual factor gives
\begin{equation}
	F_s(\phi_p, \phi_\nfa, x) 
	= 
	- \displaystyle \int_{\Gamma_b} f_s(\phi_p, \phi_\nfa, x) \d\sigma
	.
\end{equation}
For more details we refer to \cite{Bergermann:2019:1}. 
At the top surface, we assume that evaporating solvent yields an increase in polymer and \nfa~concentration, see \cite{WodoGanapathysubramanian:2012:1} for details. 
This results in a non-homogeneous flux boundary condition with right hand side equal to the product $-k\phi_p\phi_s$ and $-k\phi_\nfa \phi_s$, respectively, where $k>0$ denotes a proportionality constant.
The starting point to derive equations for the time evolution of the scalar fields is the continuity equation which results from the local conservation of mass. 
We obtain
\begin{equation}
	\label{eq:cont}
	\frac{\partial \phi_i}{\partial t} 
	= 
	- \div \bJ_i
	,
	\quad 
	i \in \{p, \, \nfa, \, s\}
	.
\end{equation}
The thermodynamic force driving the evolution is the gradient of the respective variation of the free energy.
For simplicity, we assume a linear relation between flux and force, making the additional simplification that the coefficients are constant and that off-diagonal terms are zero, which yields 
\begin{equation}
	\label{eq:flux}
	\bJ_i 
	= 
	- M_i \nabla \mu_i
	,
	\quad 
	i \in \{p, \, \nfa, \, s\}
	,
\end{equation}
where we call $M_i$ the mobility coefficient. 
We remark that due to this simplification, the relation $\phi_p+\phi_\nfa+\phi_s=1$ is not rigorously satisfied in the evolution of the phase fields. However, as soon as segregation starts, is will hold approximately and thus we still chose to solve the system \eqref{eq:cont}--\eqref{eq:flux} only for the unknowns $\phi_p$ and $\phi_\nfa$ and compute $\phi_s$ via $\phi_s = 1 - \phi_\nfa - \phi_p$.

The resulting system in strong form reads
\begin{subequations}
	\label{eq:chternary:1}
	\begin{alignat}{4}
		\frac{\partial \phi_p}{\partial t} 
		&
		= 
		\div (M_p \nabla \mu_p) 
		&
		&
		\quad
		\text{and}
		\quad
		&
		\mu_p 
		&
		= 
		\frac{\partial f}{\partial \phi_p} - \epsilon^2 \nabla^2 \phi_p 
		& 
		&
		\quad
		\text{in } 
		\Omega \times [0,T]
		, 
		\\
		\frac{\partial \phi_\nfa}{\partial t} 
		&
		= 
		\div (M_\nfa \nabla \mu_\nfa) 
		&
		&
		\quad
		\text{and}
		\quad
		&
		\mu_\nfa 
		&
		= 
		\frac{\partial f}{\partial \phi_\nfa} - \epsilon^2 \nabla^2\phi_\nfa 
		& 
		&
		\quad
		\text{in } 
		\Omega \times [0,T]
		,
	\end{alignat}
\end{subequations}
with boundary conditions
\begin{subequations}
	\label{eq:chternary:2}
	\begin{alignat}{4}
		M_p \nabla \mu_p \cdot \bn 
		&
		= 
		\begin{cases}
			\frac{\partial f_s}{\partial \phi_p} 
			&
			\text{on }
			\Gamma_b \times [0,T]
			,
			\\
			-k \phi_p \phi_s
&
\text{on }
\Gamma_t \times [0,T]
,
\\			
			0
			&
			\text{on }
			\partial\Omega \setminus (\Gamma_b \cup \Gamma_t) \times [0,T]
			,
		\end{cases}
		\\
		M_\nfa \nabla \mu_\nfa \cdot \bn 
		&
		= 
		\begin{cases}
			\frac{\partial f_s}{\partial \phi_\nfa} 
			&
			\text{on }
			\Gamma_b \times [0,T]
			,
			\\
			-k \phi_\nfa \phi_s
			&
			\text{on }
			\Gamma_t \times [0,T]
			,
			\\			
			0
			&
			\text{on }
			\partial\Omega \setminus (\Gamma_b \cup \Gamma_t) \times [0,T]
			,
		\end{cases}
	\end{alignat}
\end{subequations}
and initial conditions
\begin{subequations}
	\label{eq:chternary:3}
	\begin{alignat}{2}
	\phi_p(x,0) 
	&
	= 
	\phi_p^0(x) 
	& 
	&
	\quad
	\text{in } 
	\Omega
	,
	\\
	\phi_\nfa(x,0) 
	&
	= 
	\phi_\nfa^0(x) 
	&
	&
	\quad
	\text{in } 
	\Omega
	.
	\end{alignat}
\end{subequations}
with $\bn$ denoting the outward unit normal vector of $\Omega$.
A sensible relationship between the two sets of equations is ensured by the mutually used potential term $f(\phi_p, \phi_\nfa, \phi_s)$. 
Without it, the solution of \eqref{eq:chternary:1}--\eqref{eq:chternary:3} would contain two independent solutions without any constraints like \eqref{eq:sumphis}.
Note that due to the boundary conditions from \eqref{eq:chternary:2} the mass conservation relation \eqref{eq:sumphis} may locally be violated in the vicinity of $\Gamma_t$ and $\Gamma_b$.
However, as one moves away from the boundaries, the relation is restored by the bulk potential from \eqref{eq:logarithmicpotential}.

\begin{remark}[Existence of solutions]
	In view of the constant mobility, we expect (global-in-time) existence of weak solutions shown using a fixed point argument as in \cite{ElliottLuckhaus:1991:1}. 
	The additional difficulty here are the non-homogeneous boundary conditions and their impact on energy equalities. 
	In particular, it would be interesting to see whether the model can still be understood as a gradient flow. 
	We leave this question as future research.
\end{remark}

\section{Discretization and Preconditioning}
\label{section:discretization_preconditioning}

The discretization of the model \eqref{eq:chternary:1}--\eqref{eq:chternary:3} relies on a finite element approach \cite{BrennerScott:2008:2,ZienkiewiczMorgan:2006:1}. 
A more detailed derivation of the conversion of the infinite--dimensional continuous system into a finite dimensional non-linear system is given in \cite{Bergermann:2019:1,WodoGanapathysubramanian:2012:1}. 
Applying the finite element method gives us the above system discretized in space and we then need to consider a temporal discretization. 
We here decide on a semi-implicit scheme where the linear part of the right hand side of \eqref{eq:chternary:1} is treated implicitly and the nonlinear terms coming from the potential are handled explicitly. 
We then obtain the following system 
\begin{subequations}
	\label{eq:chternary_space}
	\begin{alignat}{3}
		\frac{1}{\tau} \bM\paren[big](){\bphi_p^{(k+1)} - \bphi_p^{(k)}}
		&
		= 
		- \bK\bmu_p^{(k+1)} 
		&
		&
		\quad
		\text{and}
		\quad
		&
		\bM \bmu_p^{(k+1)} 
		&
		= 
		\bf_p^{(k)} + \epsilon \bK \bphi_p^{(k+1)}
		,
		\\
		\frac{1}{\tau} \bM\paren[big](){\bphi_\nfa^{(k+1)} - \bphi_\nfa^{(k)}}
		&
		= 
		-
		\bK \bmu_\nfa^{(k+1)} 
		&
		&
		\quad
		\text{and}
		\quad
		&
		\bM \bmu_\nfa^{(k+1)} 
		&
		= 
		\bf_\nfa^{(k)} + \epsilon \bK \bphi_\nfa^{(k+1)}
		,
	\end{alignat}
\end{subequations}
where $\bphi_p$, $\bphi_\nfa$ and $\bmu_p$, $\bmu_\nfa$ are the coefficient vectors representing the discretized chemical potentials and order parameters, respectively, and $\tau$ denotes the time step size, which is chosen between $10^{-7}$ and $10^{-4}$ in our numerical experiments.
As observed in \cite{Bosch:2016:1} and the references mentioned therein, explicit schemes often require very tight time step restrictions, often more so than implicit schemes, and the value of the time step is coupled to discretization and interface parameters. A more detailed investigation along with more sophisticated time-stepping is a topic of future research.
The matrices $\bM$ and $\bK$ denote standard mass and stiffness matrices arising from the spatial discretization with finite elements, with appropriate constants.
They coincide for both the polymer and the non--fullerene acceptor equations.
The indices $\cdot^{(k+1)}$ and $\cdot^{(k)}$ stand for the current and previous time step.
Also $\bf_p$, $\bf_\nfa$ are the discretized representations of $\frac{\partial f}{\partial \phi_p}$ and $\frac{\partial f}{\partial \phi_\nfa}$, respectively. 
We now collect the polymer and the non-fullerene acceptor equations as 
\begin{equation}
	\label{eq:FE_systems}
	\begin{bmatrix}
		\bM
		&
		\tau
		\bK
		\\
		- \epsilon \bK
		&
		\bM
	\end{bmatrix}
	\begin{bmatrix}
		\bphi_p^{(k+1)}
		\\
		\bmu_p^{(k+1)}
	\end{bmatrix}
	=
	\begin{bmatrix}
		\bM\bphi_p^{(k)}
		\\
		\bf_p^{(k)}
	\end{bmatrix}
	,
	\quad 
	\begin{bmatrix}
		\bM
		&
		\tau \bK
		\\
		- \epsilon \bK
		&
		\bM
	\end{bmatrix}
	\begin{bmatrix}
		\bphi_\nfa^{(k+1)}
		\\
		\bmu_\nfa^{(k+1)}
	\end{bmatrix}
	=
	\begin{bmatrix}
		\bM\bphi_\nfa^{(k)}
		\\
		\bf_\nfa^{(k)}
	\end{bmatrix}
	.
\end{equation}
We need to solve these equations repeatedly for every time step and a direct solver based on a factorization is too expensive for realistic mesh sizes in spite of the maturity of the field \cite{Davis:2004:2,DuffErismanReid:2017:1}. 
To overcome this issue we focus on iterative methods, in particular Krylov subspace methods \cite{Saad:2003:1}. 
We illustrate this on one of the systems with the other being equivalent. 
Let us consider the equivalent form of the first system in \eqref{eq:FE_systems}
\begin{equation}\label{eq:saddle_point_system}
	\begin{bmatrix}
		\bM
		&
		\tau \bK
		\\
		\tau \bK
		&
		- \frac{\tau}{\epsilon} \bM
	\end{bmatrix}
	\begin{bmatrix}
		\bphi_p^{(k+1)}
		\\
		\mu_p^{(k+1)}
	\end{bmatrix}
	=
	\begin{bmatrix}
		\bM \bphi_p^{(k)}
		\\
		- \frac{\tau}{\epsilon} \bf_p^{(k)}
	\end{bmatrix}
	,
\end{equation}
which is now a symmetric saddle point system \cite{BenziGolubLiesen:2005:1}. 
Such block systems arise in a variety of different applications such as PDE-constrained optimization or the treatment of complex symmetric linear systems. 

Motivated by \cite{MurphyGolubWathen:2000:1,ElmanSilvesterWathen:2014:1} we focus on a block-diagonal preconditioner 
\begin{equation*}
	\bP
	=
	\begin{bmatrix}
		\bM 
		&
		\bnull
		\\
		\bnull
		&
		\bS 
	\end{bmatrix}
	,
\end{equation*}
where $\bS = \frac{\tau}{\epsilon}\bM + \tau^2 \bK \bM^{-1} \bK$ is the Schur complement of the matrix 
$ \bA = \begin{bsmallmatrix} \bM & \tau \bK \\ \tau \bK & -\frac{\tau}{\epsilon} \bM \end{bsmallmatrix}$. 

We here focus on this block-diagonal matrix but other structured preconditioners for related problems have been suggested \cite{BoyanovaNeytcheva:2014:1,Bai:2010:1}, especially in the context of PDE-constrained optimization. 
They all rely on an efficient approximations of shifted stiffness matrices. 
A comparison of such preconditioning strategies can be found in \cite{AxelssonFarouqNeytcheva:2016:1}. 
Our aim is to rely on a preconditioner that can be used in a symmetric Krylov subspace method such as  \minres \cite{PaigeSaunders:1975:1}. 
For this we now focus on giving further details how to approximate the Schur-complement efficiently.

The preconditioner $\bP$ is an ideal version as it is too expensive to use in practice.
We derive a practical version of this via using the approximations $\widetilde{\bM} \approx \bM$ and $\widetilde{\bS} \approx \bS$. 
The approximation $\widetilde{\bM}$ is the simpler of the two and we will use an algebraic multigrid (AMG) \cite{Falgout:2006:1} for our implementation, but also other methods such as a Chebyshev semi-iteration \cite{WathenRees:2008:2} are possible. 
The approximation of the Schur complement is more involved. 
We here follow the matching approach of \cite{PearsonWathen:2011:1,PearsonStollWathen:2012:1} given by
\begin{equation*}
	\bS
	=
	\frac{\tau}{\epsilon} \bM + \tau^2 \bK \bM^{-1} \bK
	\approx
	\paren[big](){\tau \bK + \widehat{\bM}} \bM^{-1} \paren[big](){\tau \bK + \widehat{\bM}} 
\end{equation*}
and the condition
\begin{equation*}
	\frac{\tau}{\epsilon} \bM
	=
	\widehat{\bM} \bM^{-1} \widehat{\bM}
	.
\end{equation*}
This is obviously true for $\widehat{\bM} = \sqrt{\frac{\tau}{\epsilon}} \bM$ and then the Schur complement approximation is based on 
\begin{equation*}
\widetilde{\bS} 
\coloneqq
\paren[auto](){\tau \bK + \textstyle \sqrt{\frac{\tau}{\epsilon}} \bM} \bM^{-1} \paren[auto](){\tau \bK + \textstyle \sqrt{\frac{\tau}{\epsilon}} \bM}
.
\end{equation*}
The preconditioner requires the approximate solution of systems with $\widetilde{\bS}$, which we will realize using an algebraic multigrid approximation for $\paren[big](){\tau \bK + \sqrt{\frac{\tau}{\epsilon}}\bM}$. 
Let us motivate why this preconditioner is a sensible choice (\cite{PearsonWathen:2011:1}).
\begin{lemma}\label{lemma:eigenvalues}
	The eigenvalues of the $\widetilde{\bS}^{-1}\bS$ are contained in the interval $[\frac{1}{2},1]$, independent of all system parameters.
\end{lemma}
\begin{proof}
	To show this we consider the Rayleigh quotient
	\begin{equation*}
		\frac{v^\transp \bS \, v}{v^\transp \widetilde{\bS} \, v}
		=
		\frac{v^\transp \paren[auto](){\frac{\tau}{\epsilon}\bM + \tau^2 \bK \bM^{-1} \bK} v}{v^\transp \paren[big](){\tau \bK + \sqrt{\frac{\tau}{\epsilon}} \bM} \bM^{-1} \paren[big](){\tau \bK + \sqrt{\frac{\tau}{\epsilon}} \bM} \, v}
		=
		\frac{a^\transp a + b^\transp b}{a^\transp a + b^\transp b + 2a^\transp b}
	\end{equation*}
	with $a = \tau \bM^{-1/2} \bK v$ and $b = \sqrt{\frac{\tau}{\epsilon}} \bM^{1/2} v$ and $v \neq 0$.
	We know that $0 \le \norm{a-b}^2 = (b-a)^\transp (b-a) = b^\transp b + a^\transp a - 2b^\transp a$ and as a result $2b^\transp a \le b^\transp b + a^\transp a$. 
	This gives the lower bound of $\frac{1}{2}$ for the eigenvalues. 
	The upper bound of $1$ can now be obtained if $b^\transp a \ge 0$. 
	Looking at this term in detail we see $b^\transp a = \tau \sqrt{\frac{\tau}{\epsilon}} v^\transp \bK \, v$, which is obviously positive given the property of the stiffness matrix. 
	We have thus obtained the eigenvalue bound, which is independent of the system parameters like $\epsilon$ as well as time step and mesh size.
\end{proof}
As a results we now obtain the practical version of our preconditioner as shown in \cref{algorithm:precon}. 

\begin{algorithm}
	\caption{Preconditioner}\label{algorithm:precon}
	\begin{algorithmic}[1]
		\Procedure{Prec}{$v$}
		\Comment{Application of preconditioner to block vector $v$}
		\State Approximately solve $\bM w_1 = v_1$ via AMG
		\State Approximately solve $\paren[big](){\tau \bK + \sqrt{\frac{\tau}{\epsilon}} \bM}w_2 = v_2$ via AMG
		\State Compute $w_2 \gets \bM w_2$
		\State Approximately solve $\paren[big](){\tau \bK + \sqrt{\frac{\tau}{\epsilon}} \bM} w_2 = w_2$ via AMG
		\State \textbf{return} $w$
		\Comment{Preconditioned block-vector $w$}
		\EndProcedure
	\end{algorithmic}
\end{algorithm}

\section{Numerical Experiments}
\label{section:numerical_experiments}

We solve the model \eqref{eq:chternary:1}--\eqref{eq:chternary:3} with a \python implementation using the finite element libraries \namemd{DOLFINx}\footnote{\url{https://github.com/FEniCS/dolfinx}} and \ufl \cite{AlnaesLoggOlgaardRognesWells:2014:1,AlnaesLoggMardal:2012:1} from the \fenics project \cite{AlnaesBlechtaHakeJohanssonKehletLoggRichardsonRingRognesWells:2015:1,AlnaesLoggMardal:2012:1} (latest software versions as of January~2022).
Codes that reproduce the numerical experiments presented in this section are publicly available.\footnote{\url{https://github.com/KBergermann/Precond-Cahn-Hilliard-OSC}}

We generate uniform triangulations of $2$- and $3$-dimensional rectangular domains of size $10 \times 2.5$ and $10 \times 2.5 \times 10$, respectively, with a variable number of grid points $n_x \times n_y$ and $n_x \times n_y \times n_z$ and choose linear triangular Lagrange elements. 
While we here use only regular meshes the preconditioning strategy proposed in this paper remains applicable for different meshes such as the ones obtained from an adaptive finite element scheme.
Note that the $y$-coordinate denotes the direction of the height of the film.

Boundary conditions on the top boundary $\Gamma_t$, \ie, $y = y_{\textup{max}}$ as well as the bottom (substrate) boundary $\Gamma_b$, \ie, $y = 0$ are implemented via \url surface integral measures.
On $\Gamma_b$, we enable space-dependent substrate patterning.
If activated, we have
\begin{equation*}
p_p(x) =
\begin{cases}
	0
	, 
	& 
	x \in [\frac{1}{6} x_{\textup{max}}, \frac{1}{3} x_{\textup{max}}] 
	\cup 
	[\frac{1}{2} x_{\textup{max}}, \frac{2}{3} x_{\textup{max}}] 
	\cup 
	[\frac{5}{6} x_{\textup{max}}, x_{\textup{max}}]
	,
	\\
	1
	, 
	& 
	x \in [0, \frac{1}{6} x_{\textup{max}}] 
	\cup 
	[\frac{1}{3} x_{\textup{max}}, \frac{1}{2} x_{\textup{max}}] 
	\cup 
	[\frac{2}{3} x_{\textup{max}}, \frac{5}{6} x_{\textup{max}}]
\end{cases}
\end{equation*}
and
\begin{equation*}
	p_\nfa(x) 
	=
	\begin{cases}
		0
		, 
		& 
		x \in [0, \frac{1}{6} x_{\textup{max}}] 
		\cup 
		[\frac{1}{3} x_{\textup{max}}, \frac{1}{2} x_{\textup{max}}] 
		\cup 
		[\frac{2}{3} x_{\textup{max}}, \frac{5}{6} x_{\textup{max}}]
		,
		\\
		1
		, 
		& 
		x \in [\frac{1}{6} x_{\textup{max}}, \frac{1}{3} x_{\textup{max}}] 
		\cup 
		[\frac{1}{2} x_{\textup{max}}, \frac{2}{3} x_{\textup{max}}] 
		\cup 
		[\frac{5}{6} x_{\textup{max}}, x_{\textup{max}}]
	\end{cases}
\end{equation*}
for $p_p(x)$ and $p_\nfa(x)$ in \eqref{eq:surfenergy}, \ie, polymer and non-fullerene acceptor preference alternates on equispaced sub-intervals of $\Gamma_b$.
In the absence of substrate patterning, we set $p_p(x) = p_\nfa(x) = 1$ for all $x\in\Gamma_b$.
The remaining parameters in \eqref{eq:surfenergy} are chosen as $g_p = g_\nfa = 0.01$ and $h_p = h_\nfa = 0$.
Furthermore, we choose the evaporation rate of solvent at the top boundary $\Gamma_t$, which is modeled by an inward flow of polymer and non-fullerene acceptor into the system as $k = 5 \cdot 10^{-3}$.
The Flory--Huggins parameters are chosen as $\chi_{p,\nfa} = 1, \chi_{p,s} = \chi_{\nfa,s} = 0.3$, the employed degrees of polymerization are $N_p = N_\nfa = 20$ and $N_s = 1$, and the interface parameters were chosen as $\epsilon_p = \epsilon_\nfa = 10^{-3}$.
These parameter choices are motivated by \cite{WodoGanapathysubramanian:2012:1}.

As initial conditions, we set polymer and non-fullerene acceptor concentrations to $\phi_p = \phi_\nfa = 0.35 \pm 0.01$, where $\pm 0.01$ denotes uniformly distributed random fluctuations.
These induce non-zero concentration gradients at time $t=0$, which are required for the initiation of phase separation.
Note that our model described in \Cref{section:model} omits stochasticity in the equations \eqref{eq:chternary:1}, which could be included to model noise, \cf, \eg, \cite{WodoGanapathysubramanian:2012:1}.

As in \cite{Bergermann:2019:1}, we replace the logarithmic bulk energy term \eqref{eq:logarithmicpotential} by the polynomial approximation
\begin{equation}\label{eq:polynomial_potential}
	f(\phi_p, \phi_\nfa, \phi_s) 
	= 
	3.5 \, \phi_p^2 \phi_\nfa^2 + 0.1 \, \phi_s^2
	,
\end{equation}
to allow for numerical stability in case concentrations become slightly negative due to discretization errors, which would lead to a numerical breakdown in the evaluation of the logarithmic terms.

After the linear systems \eqref{eq:FE_systems} have been assembled for both the polymer and non-fullerene acceptor equation by means of suitable \ufl expressions, all matrices are converted into \texttt{scipy.sparse} format \cite{VirtanenGommersOliphantHaberlandReddyCournapeauBurovskiPetersonWeckesserBrightVanDerWaltBrettWilsonMillmanMayorovNelsonJonesKernLarsonCareyPolatFengMooreVanderPlasLaxaldePerktoldCimrmanHenriksenQuinteroHarrisArchibaldRibeiroPedregosaVanMulbregt:2020:1}.
The AMG preconditioning steps in \Cref{algorithm:precon} are realized by the Ruge--Stüben implementation of the \python package \pyamg \cite{OlSc2018}.
\Cref{algorithm:precon} is then passed as preconditioner to the preconditioned \minres \cite{PaigeSaunders:1975:1} implementation of \scipy \cite{VirtanenGommersOliphantHaberlandReddyCournapeauBurovskiPetersonWeckesserBrightVanDerWaltBrettWilsonMillmanMayorovNelsonJonesKernLarsonCareyPolatFengMooreVanderPlasLaxaldePerktoldCimrmanHenriksenQuinteroHarrisArchibaldRibeiroPedregosaVanMulbregt:2020:1}.

\Cref{figure:minres_iterations_runtime,figure:minres_iterations_runtime_1000_iterations,figure:minres_iterations_runtime_3d,figure:minres_iterations_runtime_varying_epsilon,figure:minres_iterations_runtime_varying_dt} illustrate the required number of \minres iterations as well as the runtime required to solve one preconditioned linear system of the form \eqref{eq:saddle_point_system} to a tolerance of $10^{-7}$ in terms of the concentrations $\phi_p, \phi_\nfa, \phi_s\in[0,1]$.
All numerical experiments were performed on an AMD Ryzen~5 5600X 6-Core processor with 16~GiB RAM.
The experiments corroborate the theoretical result from \Cref{lemma:eigenvalues}, \ie, the independence of the preconditioner from \Cref{algorithm:precon} of system parameters as well as time step and mesh sizes.

More specifically, \Cref{figure:minres_iterations_runtime,figure:minres_iterations_runtime_3d} indicate relatively constant \minres iteration numbers as well as an almost perfect linear scaling of the runtime in the first $20$ time steps for $2$- and $3$-dimensional meshes of different fineness.
In the 2D case, \Cref{figure:minres_iterations_runtime_1000_iterations} shows that very similar observations hold true after $1000$~time steps have been completed.
Furthermore, \Cref{figure:minres_iterations_runtime_varying_epsilon,figure:minres_iterations_runtime_varying_dt} illustrate that \minres iterations and runtimes are similar for different values of the interface parameters $\epsilon_p$ and $\epsilon_\nfa$ as well as the chosen time step size $\tau$ for parameter ranges that permit stable numerical solutions.

In order to numerically verify the theoretically indicated linear convergence order of the IMEX scheme we ran our solver for five different time step sizes between $\tau=2 \cdot 10^{-4}$ and $3.2 \cdot 10^{-3}$.
Comparing the results to a relatively fine solution computed with $\tau=10^{-5}$ we obtained an experimental order of convergence of $0.789$ for the polymer and $0.788$ for the non-fullerene acceptor equation.

Finally, \Cref{figure:2d_morphologies,figure:3d_morphologies} show simulation results for both the $2$- and $3$-dimensional case.
Details on the parameter choices are given in the respective captions.

\begin{figure}[htp]
	\subfloat[\minres iterations]{
		\begin{tikzpicture}[baseline]
			\begin{axis}[
				xlabel={time step},
				ylabel={\# iterations},
				width=0.45\linewidth,
				height=150pt,
				legend style={at={(0.4, -0.3)}, anchor=north, legend columns=3},
				]
				
				\addplot[-red!90, mark=o] coordinates {
					(1, 14) (2, 14) (3, 14) (4, 14) (5, 15) (6, 15) (7, 15) (8, 15) (9, 15) (10, 15) (11, 15) (12, 15) (13, 15) (14, 15) (15, 15) (16, 15) (17, 15) (18, 15) (19, 15) (20, 15)
				};
				\addlegendentry{$100 \times 50$}
				
				\addplot[-red!60!green!60, mark=x] coordinates {
					(1, 14) (2, 14) (3, 14) (4, 14) (5, 14) (6, 14) (7, 14) (8, 14) (9, 14) (10, 14) (11, 14) (12, 14) (13, 14) (14, 14) (15, 14) (16, 14) (17, 14) (18, 14) (19, 14) (20, 14)
				};
				\addlegendentry{$200 \times 100$}
				
				\addplot[-red!60!green!30, mark=square] coordinates {
					(1, 14) (2, 14) (3, 14) (4, 14) (5, 14) (6, 14) (7, 14) (8, 14) (9, 14) (10, 14) (11, 14) (12, 14) (13, 14) (14, 14) (15, 14) (16, 14) (17, 14) (18, 14) (19, 14) (20, 14)
				};
				\addlegendentry{$400 \times 200$}
				
				\addplot[red!60!green!30, mark=o] coordinates {
					(1, 12) (2, 12) (3, 12) (4, 12) (5, 12) (6, 12) (7, 12) (8, 12) (9, 12) (10, 13) (11, 13) (12, 13) (13, 13) (14, 13) (15, 13) (16, 13) (17, 13) (18, 13) (19, 13) (20, 13)
				};
				\addlegendentry{$800 \times 400$}
				
				\addplot[red!60!green!60, mark=x] coordinates {
					(1, 12) (2, 12) (3, 11) (4, 11) (5, 11) (6, 11) (7, 11) (8, 11) (9, 11) (10, 11) (11, 11) (12, 11) (13, 11) (14, 11) (15, 12) (16, 12) (17, 12) (18, 12) (19, 12) (20, 12)
				};
				\addlegendentry{$1600 \times 800$}
				
				\addplot[red!90, mark=square] coordinates {
					(1, 10) (2, 10) (3, 10) (4, 10) (5, 10) (6, 10) (7, 10) (8, 10) (9, 10) (10, 10) (11, 10) (12, 10) (13, 10) (14, 10) (15, 10) (16, 10) (17, 10) (18, 10) (19, 10) (20, 10)
				};
				\addlegendentry{$3200 \times 1600$}
				
			\end{axis}
		\end{tikzpicture}
	}
	\hfill
	\subfloat[\minres runtimes]{
		\begin{tikzpicture}[baseline]
			\begin{axis}[
				xmode=log,
				ymode=log,
				xlabel={degrees of freedom},
				ylabel={runtime in seconds},
				width=0.45\linewidth,
				height=150pt,
				]
				
				\addplot[black,mark=square,mark size=2,mark options={solid}] coordinates {
					(5000, 0.089) (20000, 0.23) (80000, 1.03) (320000, 4.06) (1280000, 18.8) (5120000, 84.3)
				};
				
			\end{axis}
		\end{tikzpicture}
	}
	\caption{%
		\minres iteration numbers and runtimes per time step for different spatial discretizations of a rectangular 2D domain ($n_x \times n_y$ denotes the number of spatial discretization points in $x$ and $y$ direction, respectively).
		\minres tolerance $10^{-7}$, AMG preconditioner tolerance $10^{-4}$, scaled time step size $\tau = 10^{-4}$, interface parameters $\epsilon_\nfa = \epsilon_p = 10^{-3}$.
		All numbers indicate the worst observed case, \ie, the larger number of iterations of the two (polymer and non-fullerene acceptor) equations and the longest observed runtime of both equations over the first $20$ time steps.
	}
	\label{figure:minres_iterations_runtime}
\end{figure}
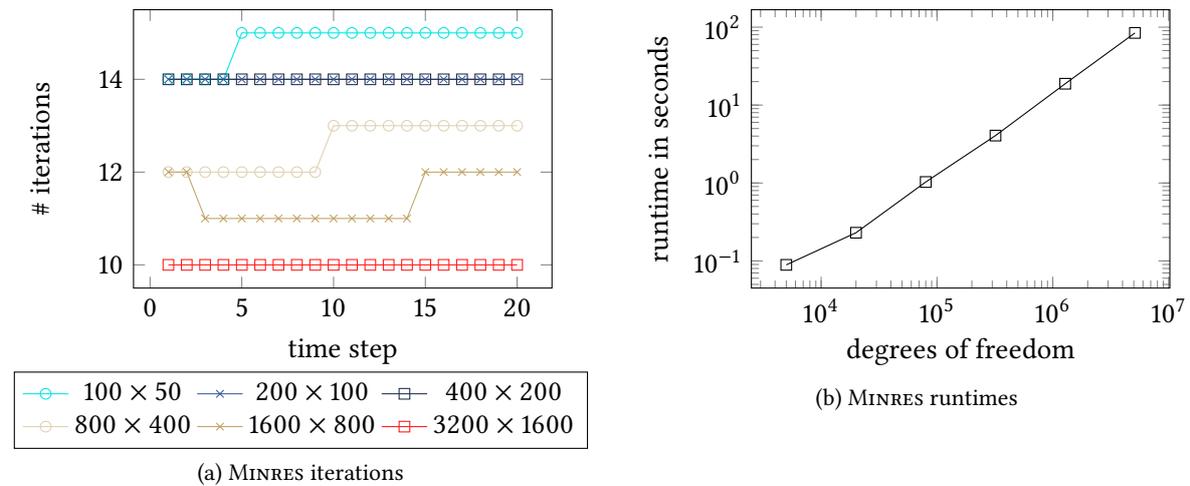

\begin{figure}[htp]
	\subfloat[\minres iterations]{
		\begin{tikzpicture}[baseline]
		\begin{axis}[
		xlabel={time step},
		ylabel={\# iterations},
		width=0.45\linewidth,
		height=150pt,
		legend style={at={(0.4, -0.3)}, anchor=north, legend columns=3},
		]
		
		\addplot[-red!90, mark=o] coordinates {
			(1, 15) (2, 15) (3, 15) (4, 15) (5, 15) (6, 15) (7, 15) (8, 15) (9, 15) (10, 15) (11, 15) (12, 15) (13, 15) (14, 15) (15, 15) (16, 15) (17, 15) (18, 15) (19, 15) (20, 15)
		};
		\addlegendentry{$100 \times 50$}
		
		\addplot[-red!60!green!60, mark=x] coordinates {
			(1, 16) (2, 16) (3, 16) (4, 16) (5, 16) (6, 16) (7, 16) (8, 16) (9, 16) (10, 16) (11, 16) (12, 16) (13, 16) (14, 16) (15, 16) (16, 16) (17, 16) (18, 16) (19, 16) (20, 16)
		};
		\addlegendentry{$200 \times 100$}
		
		\addplot[-red!60!green!30, mark=square] coordinates {
			(1, 15) (2, 15) (3, 15) (4, 15) (5, 15) (6, 15) (7, 15) (8, 15) (9, 15) (10, 15) (11, 15) (12, 15) (13, 15) (14, 15) (15, 15) (16, 15) (17, 15) (18, 15) (19, 15) (20, 15)
		};
		\addlegendentry{$400 \times 200$}
		
		\addplot[red!60!green!30, mark=o] coordinates {
			(1, 14) (2, 14) (3, 14) (4, 14) (5, 14) (6, 14) (7, 14) (8, 14) (9, 14) (10, 14) (11, 14) (12, 14) (13, 14) (14, 14) (15, 14) (16, 14) (17, 14) (18, 14) (19, 14) (20, 14)
		};
		\addlegendentry{$800 \times 400$}
		
		\addplot[red!60!green!60, mark=x] coordinates {
			(1, 13) (2, 13) (3, 13) (4, 13) (5, 13) (6, 13) (7, 13) (8, 13) (9, 13) (10, 13) (11, 13) (12, 13) (13, 13) (14, 13) (15, 13) (16, 13) (17, 13) (18, 13) (19, 13) (20, 13)
		};
		\addlegendentry{$1600 \times 800$}

		\end{axis}
		\end{tikzpicture}
	}
	\hfill
	\subfloat[\minres runtimes]{
		\begin{tikzpicture}[baseline]
		\begin{axis}[
		xmode=log,
		ymode=log,
		xlabel={degrees of freedom},
		ylabel={runtime in seconds},
		width=0.45\linewidth,
		height=150pt,
		]
		
		\addplot[black,mark=square,mark size=2,mark options={solid}] coordinates {
			(5000, 0.086) (20000, 0.27) (80000, 1.06) (320000, 4.42) (1280000, 20.3) 
		};
		
		\end{axis}
		\end{tikzpicture}
	}
	\caption{%
		\minres iteration numbers and runtimes per time step after $1000$ time steps, \ie, time step $1$ in panel (a) denotes the $1001$th overall time step.
		Otherwise, setup and parameters are the same as in \Cref{figure:minres_iterations_runtime}.
	}
	\label{figure:minres_iterations_runtime_1000_iterations}
\end{figure}
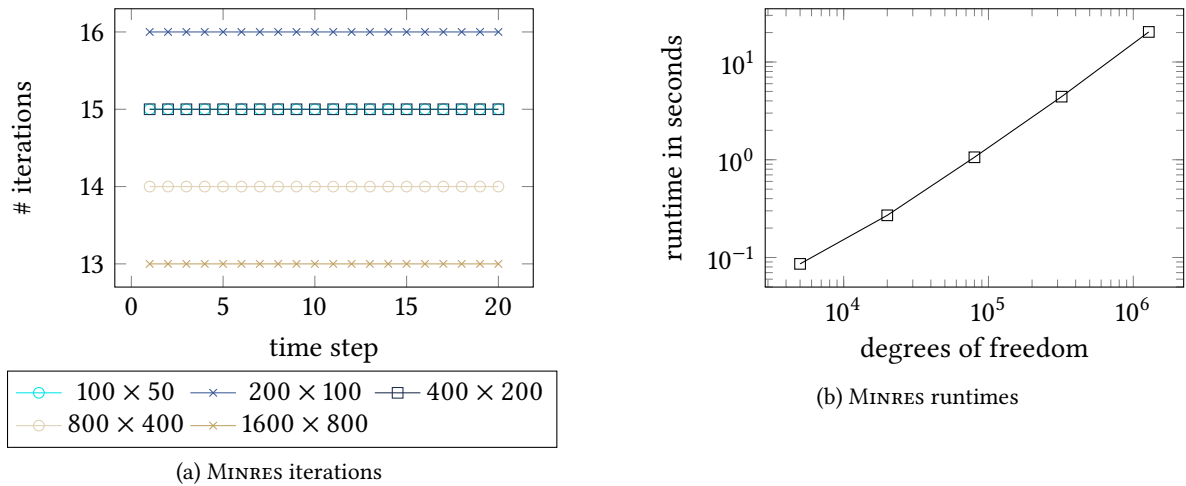

\begin{figure}[htp]
	\subfloat[\minres iterations]{
		\begin{tikzpicture}[baseline]
		\begin{axis}[
		xlabel={time step},
		ylabel={\# iterations},
		width=0.45\linewidth,
		height=150pt,
		legend style={at={(0.4, -0.3)}, anchor=north, legend columns=2},
		]
		
		\addplot[-red!90, mark=o] coordinates {
			(1, 8) (2, 8) (3, 8) (4, 8) (5, 8) (6, 8) (7, 8) (8, 8) (9, 8) (10, 8) (11, 8) (12, 8) (13, 8) (14, 8) (15, 8) (16, 8) (17, 8) (18, 8) (19, 8) (20, 8)
		};
		\addlegendentry{$30 \times 15 \times 30$}
		
		\addplot[-red!60!green!60, mark=x] coordinates {
			(1, 10) (2, 8) (3, 8) (4, 8) (5, 8) (6, 8) (7, 8) (8, 8) (9, 8) (10, 8) (11, 8) (12, 8) (13, 8) (14, 8) (15, 8) (16, 8) (17, 8) (18, 8) (19, 8) (20, 9)
		};
		\addlegendentry{$40 \times 20 \times 40$}
		
		\addplot[-red!60!green!30, mark=square] coordinates {
			(1, 12) (2, 11) (3, 10) (4, 10) (5, 10) (6, 10) (7, 10) (8, 10) (9, 10) (10, 10) (11, 10) (12, 11) (13, 11) (14, 11) (15, 11) (16, 11) (17, 11) (18, 11) (19, 11) (20, 11)
		};
		\addlegendentry{$60 \times 30 \times 60$}
		
		\addplot[red!60!green!30, mark=o] coordinates {
			(1, 12) (2, 11) (3, 11) (4, 11) (5, 11) (6, 11) (7, 11) (8, 11) (9, 11) (10, 11) (11, 11) (12, 11) (13, 11) (14, 11) (15, 11) (16, 11) (17, 11) (18, 11) (19, 11) (20, 11)
		};
		\addlegendentry{$80 \times 40 \times 80$}
		
		\addplot[red!60!green!60, mark=x] coordinates {
			(1, 12) (2, 11) (3, 11) (4, 12) (5, 12) (6, 12) (7, 12) (8, 12) (9, 12) (10, 12) (11, 12) (12, 12) (13, 12) (14, 12) (15, 12) (16, 12) (17, 12) (18, 12) (19, 12) (20, 12)
		};
		\addlegendentry{$120 \times 60 \times 120$}
		
		\addplot[red!90, mark=square] coordinates {
			(1, 12) (2, 12) (3, 12) (4, 12) (5, 12) (6, 12) (7, 12) (8, 12) (9, 12) (10, 12) (11, 12) (12, 12) (13, 12) (14, 12) (15, 12) (16, 12) (17, 12) (18, 12) (19, 12) (20, 12)
		};
		\addlegendentry{$160 \times 80 \times 160$}
		
		\end{axis}
		\end{tikzpicture}
	}
	\hfill
	\subfloat[\minres runtimes]{
		\begin{tikzpicture}[baseline]
		\begin{axis}[
		xmode=log,
		ymode=log,
		xlabel={degrees of freedom},
		ylabel={runtime in seconds},
		width=0.45\linewidth,
		height=150pt,
		]
		
		\addplot[black,mark=square,mark size=2,mark options={solid}] coordinates {
			(13500, 0.136) (32000, 0.32) (108000, 1.43) (256000, 3.16) (864000, 13.3) (2048000, 30.5)
		};
		
		\end{axis}
		\end{tikzpicture}
	}
	\caption{%
		\minres iteration numbers and runtimes per time step for different spatial discretizations of a rectangular 3D domain ($n_x \times n_y \times n_z$ denotes the number of spatial discretization points in $x$, $y$, and $z$ direction, respectively).
		Otherwise, setup and parameters are the same as in \Cref{figure:minres_iterations_runtime}.
	}
	\label{figure:minres_iterations_runtime_3d}
\end{figure}
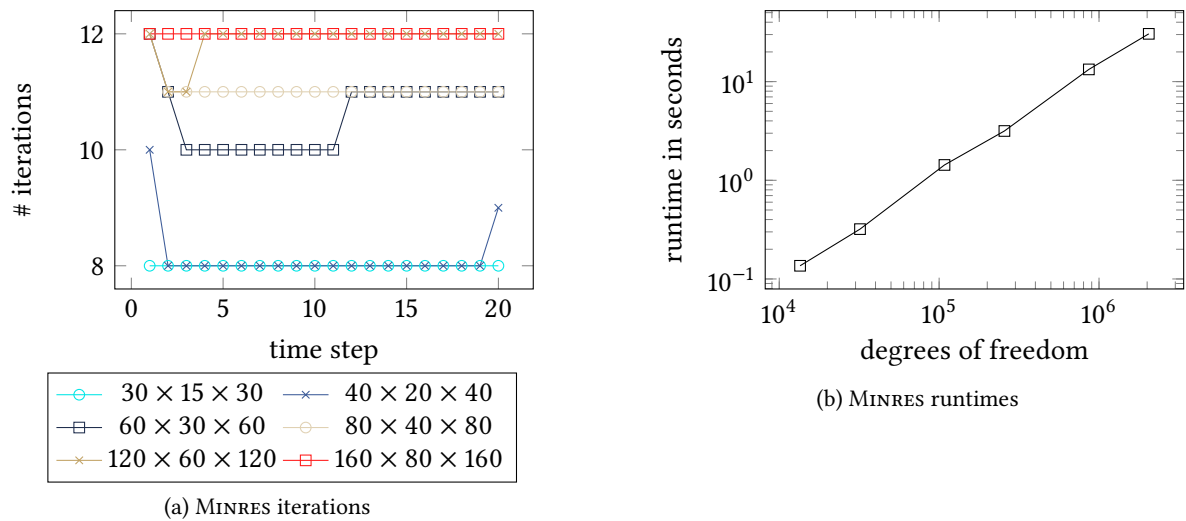

\begin{figure}[htp]
	\subfloat[\minres iterations]{
		\begin{tikzpicture}[baseline]
		\begin{axis}[
		xlabel={time step},
		ylabel={\# iterations},
		width=0.45\linewidth,
		height=150pt,
		legend style={at={(0.4, -0.3)}, anchor=north, legend columns=2},
		]
				
		\addplot[-red!90, mark=o] coordinates {
			(1, 14) (2, 15) (3, 16) (4, 16) (5, 16) (6, 16) (7, 16) (8, 17) (9, 15) (10, 1) (11, 1) (12, 1) (13, 1) (14, 1) (15, 1) (16, 1) (17, 1) (18, 1) (19, 1) (20, 1)
		};
		\addlegendentry{$\epsilon_\nfa = \epsilon_p=10^{-4}$}
		
		\addplot[-red!60!green!60, mark=x] coordinates {
			(1, 12) (2, 12) (3, 12) (4, 12) (5, 12) (6, 12) (7, 12) (8, 12) (9, 12) (10, 13) (11, 13) (12, 13) (13, 13) (14, 13) (15, 13) (16, 13) (17, 13) (18, 13) (19, 13) (20, 13)
		};
		\addlegendentry{$\epsilon_\nfa = \epsilon_p=10^{-3}$}
		
		\addplot[-red!60!green!30, mark=square] coordinates {
			(1, 10) (2, 10) (3, 10) (4, 10) (5, 10) (6, 10) (7, 10) (8, 10) (9, 10) (10, 10) (11, 10) (12, 10) (13, 10) (14, 10) (15, 10) (16, 10) (17, 10) (18, 10) (19, 10) (20, 10)
		};
		\addlegendentry{$\epsilon_\nfa = \epsilon_p=10^{-2}$}
		
		\addplot[red!60!green!30, mark=o] coordinates {
			(1, 9) (2, 8) (3, 8) (4, 8) (5, 8) (6, 8) (7, 8) (8, 8) (9, 8) (10, 8) (11, 8) (12, 8) (13, 8) (14, 8) (15, 8) (16, 8) (17, 8) (18, 8) (19, 8) (20, 8)
		};
		\addlegendentry{$\epsilon_\nfa = \epsilon_p=10^{-1}$}
		
		\addplot[red!60!green!60, mark=x] coordinates {
			(1, 8) (2, 6) (3, 6) (4, 6) (5, 5) (6, 5) (7, 5) (8, 5) (9, 5) (10, 5) (11, 5) (12, 5) (13, 5) (14, 5) (15, 5) (16, 5) (17, 5) (18, 5) (19, 5) (20, 5)
		};
		\addlegendentry{$\epsilon_\nfa = \epsilon_p=1$}
		
		\addplot[red!90, mark=square] coordinates {
			(1, 7) (2, 4) (3, 4) (4, 4) (5, 4) (6, 4) (7, 4) (8, 4) (9, 3) (10, 3) (11, 3) (12, 3) (13, 3) (14, 3) (15, 3) (16, 3) (17, 3) (18, 3) (19, 3) (20, 3)
		};
		\addlegendentry{$\epsilon_\nfa = \epsilon_p=10$}
		
		\end{axis}
		\end{tikzpicture}
	}
	\hfill
	\subfloat[\minres runtimes]{
		\begin{tikzpicture}[baseline]
		\begin{axis}[
		xmode=log,
		xlabel={$\epsilon_p = \epsilon_\nfa$},
		ylabel={runtime in seconds},
		width=0.45\linewidth,
		height=150pt,
		]
		
		\addplot[black,mark=square,mark size=2,mark options={solid}] coordinates {
			(0.0001, 4.82) (0.001, 4.06) (0.01, 3.35) (0.1, 3.49) (1, 3.92) (10, 4.15)
		};
		
		\end{axis}
		\end{tikzpicture}
	}
	\caption{%
		\minres iteration numbers and runtimes per time step for a spatial discretizations of a rectangular 2D domain of $800 \times 400$ ($n_x \times n_y$ denotes the number of spatial discretization points in $x$ and $y$ direction, respectively) with varying interface parameters $\epsilon_\nfa = \epsilon_p$.
		Otherwise, setup and parameters are the same as in \Cref{figure:minres_iterations_runtime}.
		Note that for $\epsilon_\nfa = \epsilon_p=10^{-4}$ the concentrations $\phi_p$ and $\phi_\nfa$ diverge.
	}
	\label{figure:minres_iterations_runtime_varying_epsilon}
\end{figure}
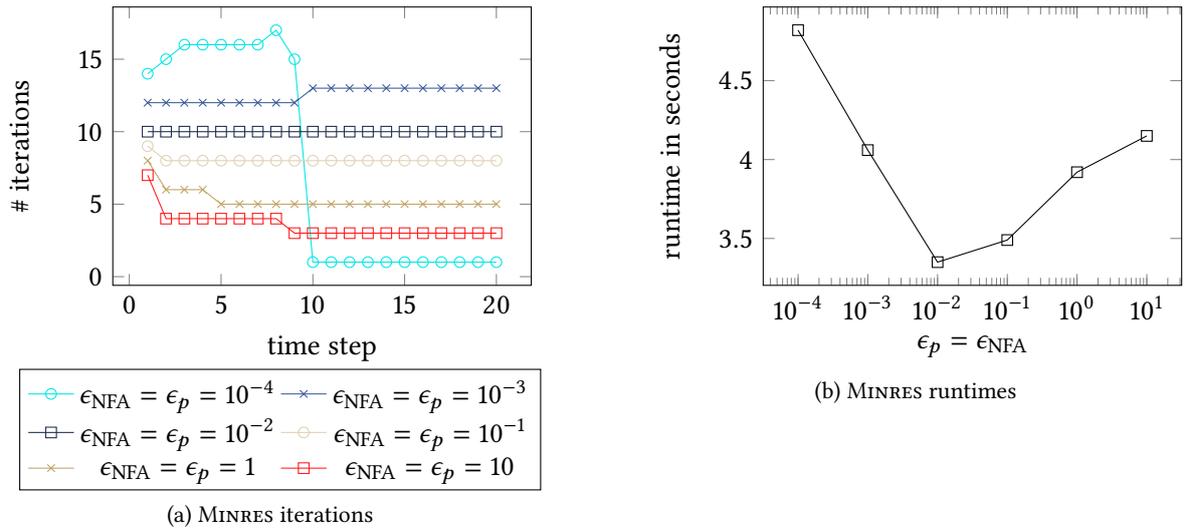

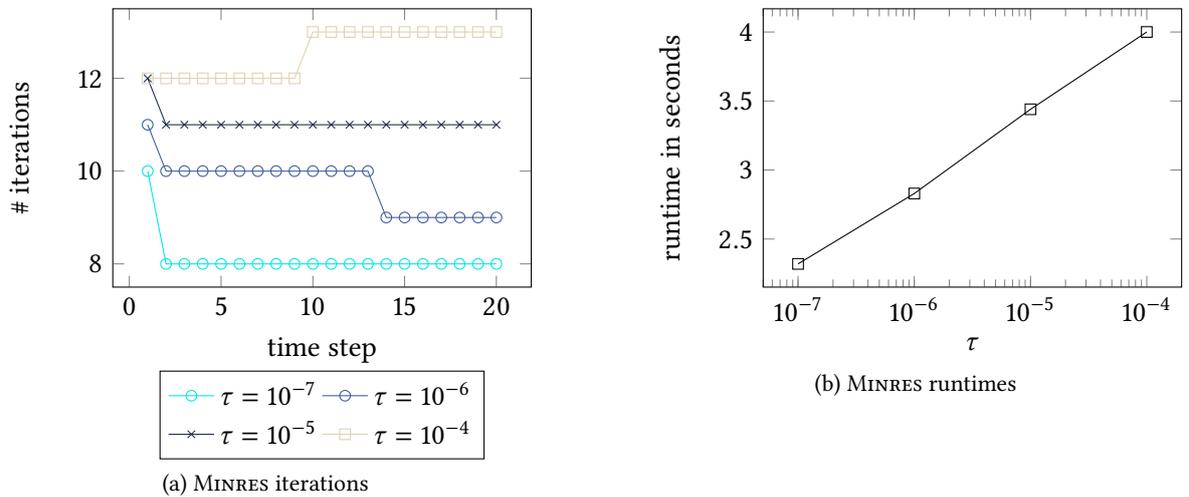
\begin{figure}[htp]
	\subfloat[\minres iterations]{
		\begin{tikzpicture}[baseline]
		\begin{axis}[
		xlabel={time step},
		ylabel={\# iterations},
		width=0.45\linewidth,
		height=150pt,
		legend style={at={(0.5, -0.3)}, anchor=north, legend columns=2},
		]
		
		\addplot[-red!90, mark=o] coordinates {
			(1, 10) (2, 8) (3, 8) (4, 8) (5, 8) (6, 8) (7, 8) (8, 8) (9, 8) (10, 8) (11, 8) (12, 8) (13, 8) (14, 8) (15, 8) (16, 8) (17, 8) (18, 8) (19, 8) (20, 8)
		};
		\addlegendentry{$\tau=10^{-7}$}
		
		\addplot[-red!60!green!60, mark=o] coordinates {
			(1, 11) (2, 10) (3, 10) (4, 10) (5, 10) (6, 10) (7, 10) (8, 10) (9, 10) (10, 10) (11, 10) (12, 10) (13, 10) (14, 9) (15, 9) (16, 9) (17, 9) (18, 9) (19, 9) (20, 9)
		};
		\addlegendentry{$\tau=10^{-6}$}
		
		\addplot[-red!60!green!30, mark=x] coordinates {
			(1, 12) (2, 11) (3, 11) (4, 11) (5, 11) (6, 11) (7, 11) (8, 11) (9, 11) (10, 11) (11, 11) (12, 11) (13, 11) (14, 11) (15, 11) (16, 11) (17, 11) (18, 11) (19, 11) (20, 11)
		};
		\addlegendentry{$\tau=10^{-5}$}
		
		\addplot[red!60!green!30, mark=square] coordinates {
			(1, 12) (2, 12) (3, 12) (4, 12) (5, 12) (6, 12) (7, 12) (8, 12) (9, 12) (10, 13) (11, 13) (12, 13) (13, 13) (14, 13) (15, 13) (16, 13) (17, 13) (18, 13) (19, 13) (20, 13)
		};
		\addlegendentry{$\tau=10^{-4}$}
		
		\end{axis}
		\end{tikzpicture}
	}
	\hfill
	\subfloat[\minres runtimes]{
		\begin{tikzpicture}[baseline]
		\begin{axis}[
		xmode=log,
		xlabel={$\tau$},
		ylabel={runtime in seconds},
		width=0.45\linewidth,
		height=150pt,
		]
		
		\addplot[black,mark=square,mark size=2,mark options={solid}] coordinates {
			(0.0000001, 2.32) (0.000001, 2.83) (0.00001, 3.44) (0.0001, 4.0)
		};
		
		\end{axis}
		\end{tikzpicture}
	}
	\caption{%
		\minres iteration numbers and runtimes per time step for a spatial discretizations of a rectangular 2D domain of $800 \times 400$ ($n_x \times n_y$ denotes the number of spatial discretization points in $x$ and $y$ direction, respectively) with varying scaled time step size $\tau$.
		Otherwise, setup and parameters are the same as in \Cref{figure:minres_iterations_runtime}.
	}
	\label{figure:minres_iterations_runtime_varying_dt}
\end{figure}

\begin{figure}[htp]
	\subfloat[]{
		\includegraphics[width=.45\textwidth]{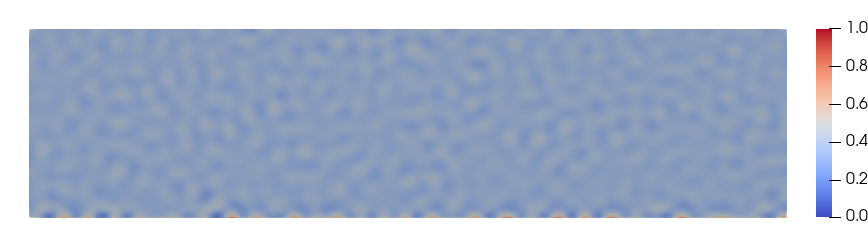}
	}
	\hfill
	\subfloat[]{
		\includegraphics[width=.45\textwidth]{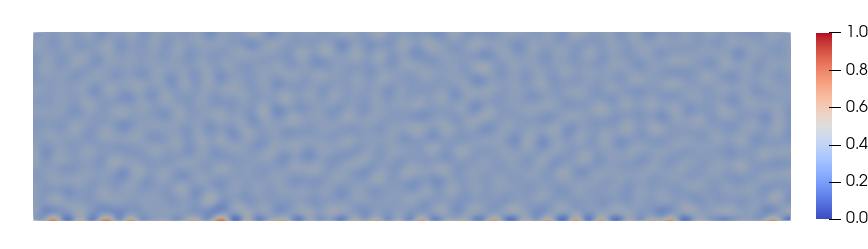}
	}
	
	\subfloat[]{
		\includegraphics[width=.45\textwidth]{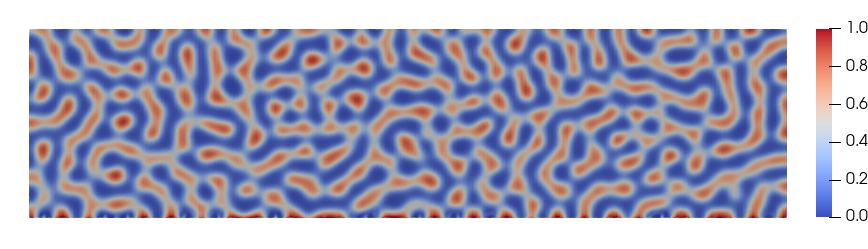}
	}
	\hfill
	\subfloat[]{
		\includegraphics[width=.45\textwidth]{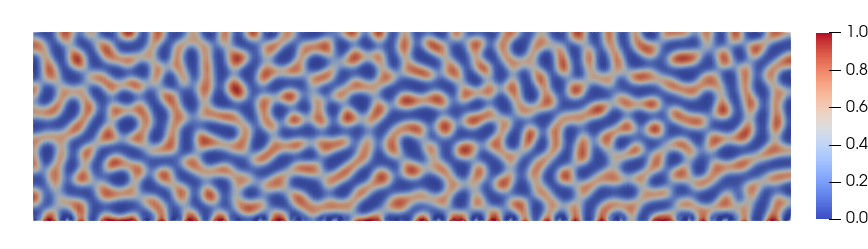}
	}
	
	\subfloat[]{
		\includegraphics[width=.45\textwidth]{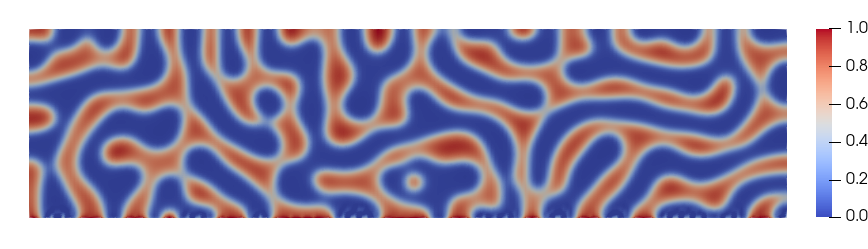}
	}
	\hfill
	\subfloat[]{
		\includegraphics[width=.45\textwidth]{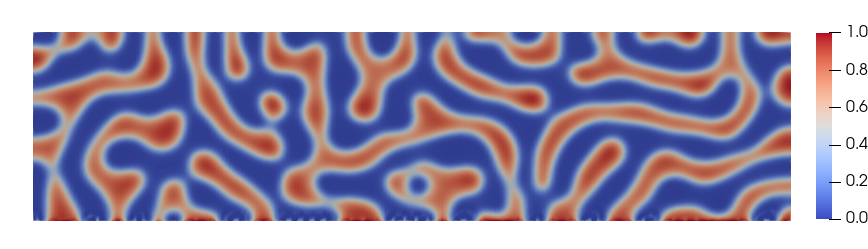}
	}
	
	\subfloat[]{
		\includegraphics[width=.45\textwidth]{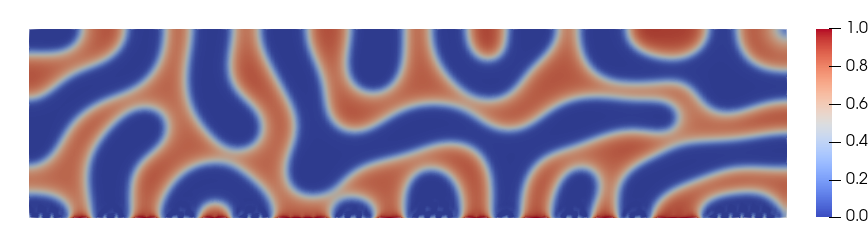}
	}
	\hfill
	\subfloat[]{
		\includegraphics[width=.45\textwidth]{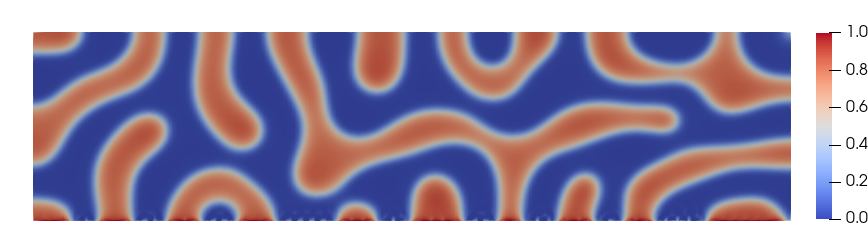}
	}
	
	\subfloat[]{
		\includegraphics[width=.45\textwidth]{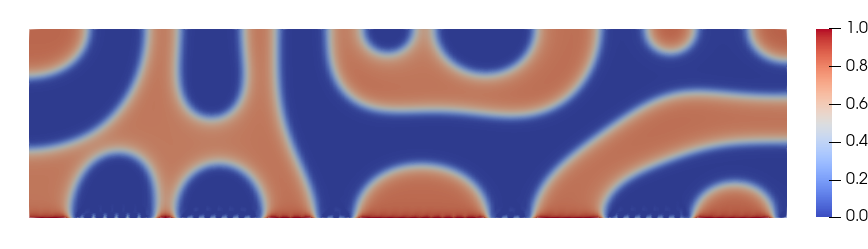}
	}
	\hfill
	\subfloat[]{
		\includegraphics[width=.45\textwidth]{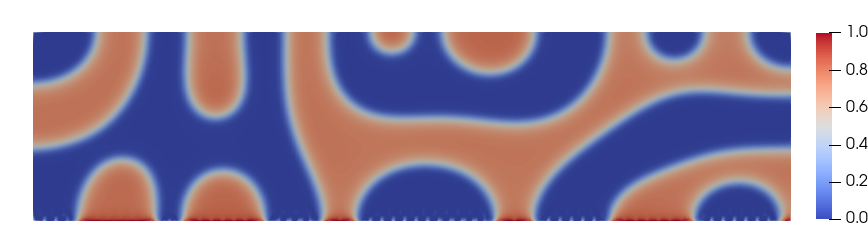}
	}
	\caption{%
		Example 2D morphology evolution at five different times without substrate patterning on the lower boundary.
		Time evolves from top to bottom. 
		The time step size $2 \cdot 10^{-4}$ was used and the depicted morphologies correspond to $t=0.02$, $t=0.06$, $t=0.4$, $t=2$, and $t=10$ (final time).
		Left: polymer concentration, right: non-fullerene acceptor concentration.
		The spatial discretization was chosen $200 \times 100$.
		Furthermore, the interface parameters $\epsilon_\nfa=\epsilon_p=10^{-3}$, initial concentrations $\phi_p=\phi_\nfa=0.35 \pm 0.01$, degrees of polymerization $N_p=N_\nfa=20, N_s=1$, Flory--Huggins interaction parameters $\chi_{p,\nfa}=1, \chi_{p,s}=\chi_{\nfa,s}=0.3$, and the polynomial approximation \eqref{eq:polynomial_potential} of the logarithmic potential were used.
	}
	\label{figure:2d_morphologies}
\end{figure}

\begin{figure}[htp]
	\subfloat[]{
		\includegraphics[width=.45\textwidth]{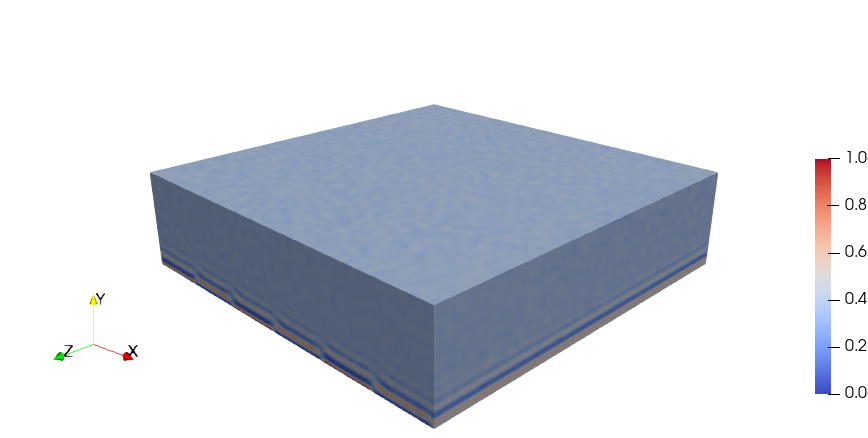}
	}
	\hfill
	\subfloat[]{
		\includegraphics[width=.45\textwidth]{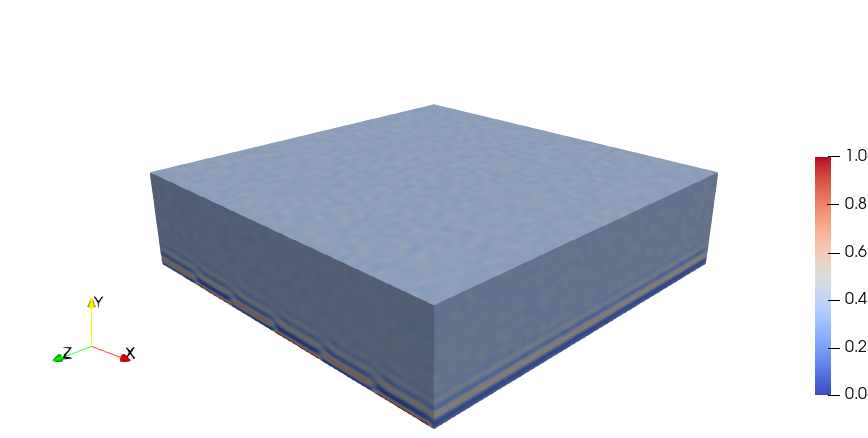}
	}
	
	\subfloat[]{
		\includegraphics[width=.45\textwidth]{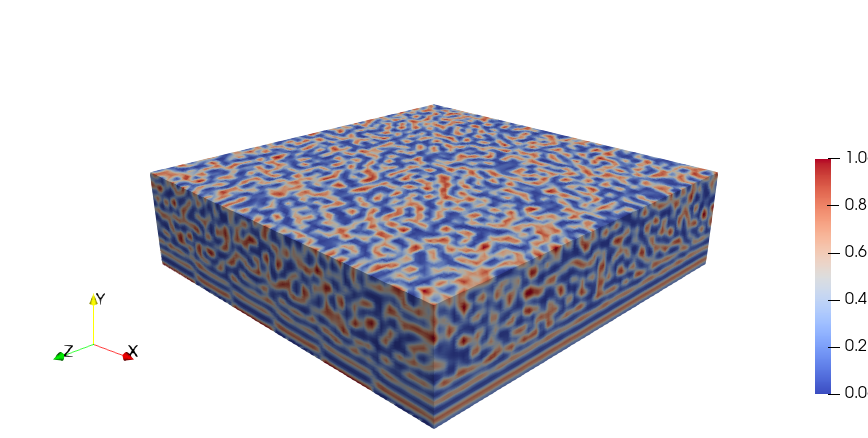}
	}
	\hfill
	\subfloat[]{
		\includegraphics[width=.45\textwidth]{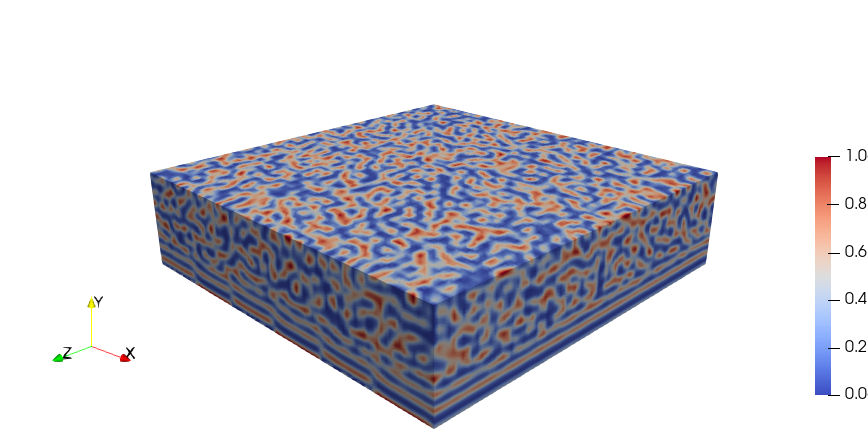}
	}
	
	\subfloat[]{
		\includegraphics[width=.45\textwidth]{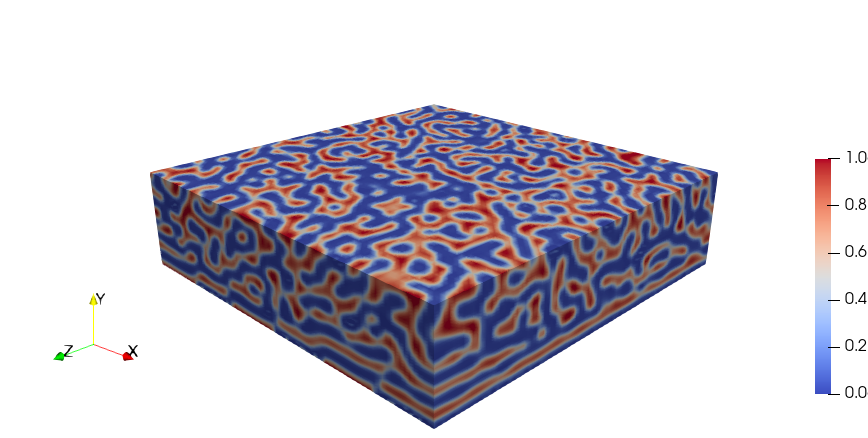}
	}
	\hfill
	\subfloat[]{
		\includegraphics[width=.45\textwidth]{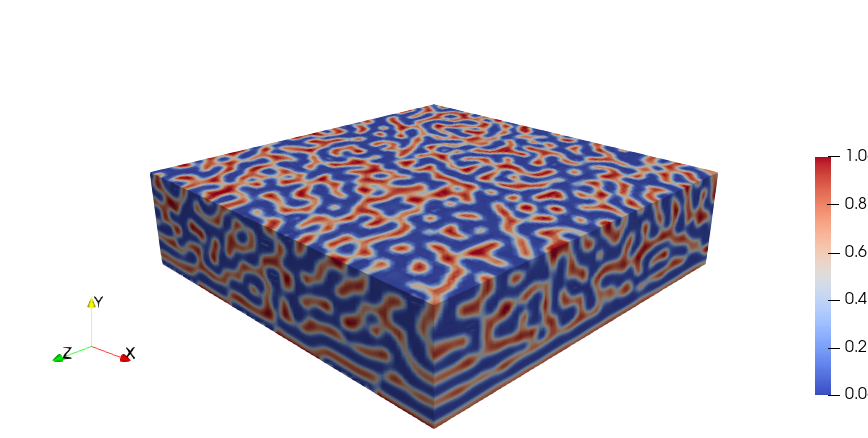}
	}
	
	\subfloat[]{
		\includegraphics[width=.45\textwidth]{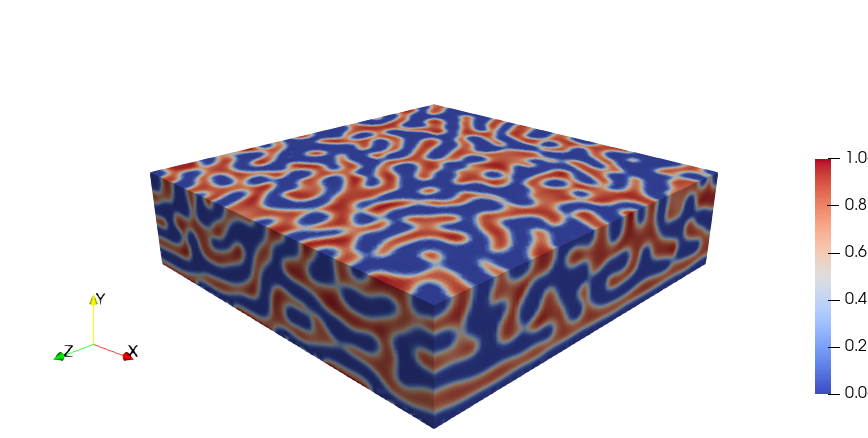}
	}
	\hfill
	\subfloat[]{
		\includegraphics[width=.45\textwidth]{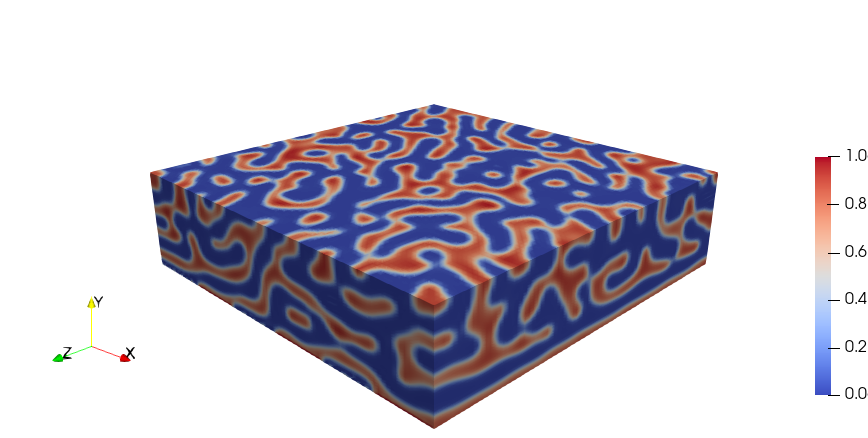}
	}
	\caption{%
		Example 3D morphology evolution at four different times with substrate patterning on the lower boundary.
		Time evolves from top to bottom. 
		The time step size $2 \cdot 10^{-4}$ was used and the depicted morphologies correspond to $t=0.012$, $t=0.04$, $t=0.24$, and $t=1$ (final time).
		Left: polymer concentration, right: non-fullerene acceptor concentration.
		The spatial discretization was chosen $80 \times 40 \times 80$.
		Furthermore, the interface parameters $\epsilon_\nfa=\epsilon_p=10^{-3}$, initial concentrations $\phi_p=\phi_\nfa=0.35 \pm 0.01$, degrees of polymerization $N_p=N_\nfa=20, N_s=1$, Flory--Huggins interaction parameters $\chi_{p,\nfa}=1, \chi_{p,s}=\chi_{\nfa,s}=0.3$, and the polynomial approximation \eqref{eq:polynomial_potential} of the logarithmic potential were used.
	}
	\label{figure:3d_morphologies}
\end{figure}

\section{Exploiting Simulated Morphologies}
\label{section:exploiting_morphologies}

This section explores how the generated morphologies link to the electrical performance of devices. 
We first used a thresholding function to convert the computationally generated morphologies into binary images. 
We then discredited these images using a regular triangular mesh, vertex removal was performed to reduce the number of triangles in the mesh, and the back of the object sealed with two more triangles to generate an enclosed object. 
The mesh describing the morphology was loaded into our 2D finite difference drift diffusion model (\url{https://www.gpvdm.com}) \cite{MaityRamananArieseMacKenzieVonHauff:2022:1}. 
The structure was then projected onto a 2D finite difference grid by shooting light rays from each mesh point on the grid to the top of the simulation world. 
If the ray intersected an odd number of faces of an object we were able to tell that the point on the grid resides within that object, if an even number of faces is encountered the point lies outside the object. 
The object with an odd number of triangles and a face closest to the light source was taken as the object associated with the 2D grid point. 
Mobility values were set depending upon which object the mesh point resided in. 
A high value of electron mobility (\SI{0.1}{\meter\squared\per\second\per\volt}) and a low value of hole mobility (\SI{1e-10}{\meter\squared\per\second\per\volt}) was assumed for the non-polymer phase, and the opposite values for the polymer phase.

The bi-polar drift-diffusion equations along with Poisson's equation were then solved on the 2D finite difference grid using a Scharfetter--Gummel discretization and Fermi--Dirac stastistics. 
The set of equations are written out in a single Jacobian and all solved together iteratively using Newton's method.

\begin{figure}[H]
	\centering
	\includegraphics[width=0.7\textwidth]{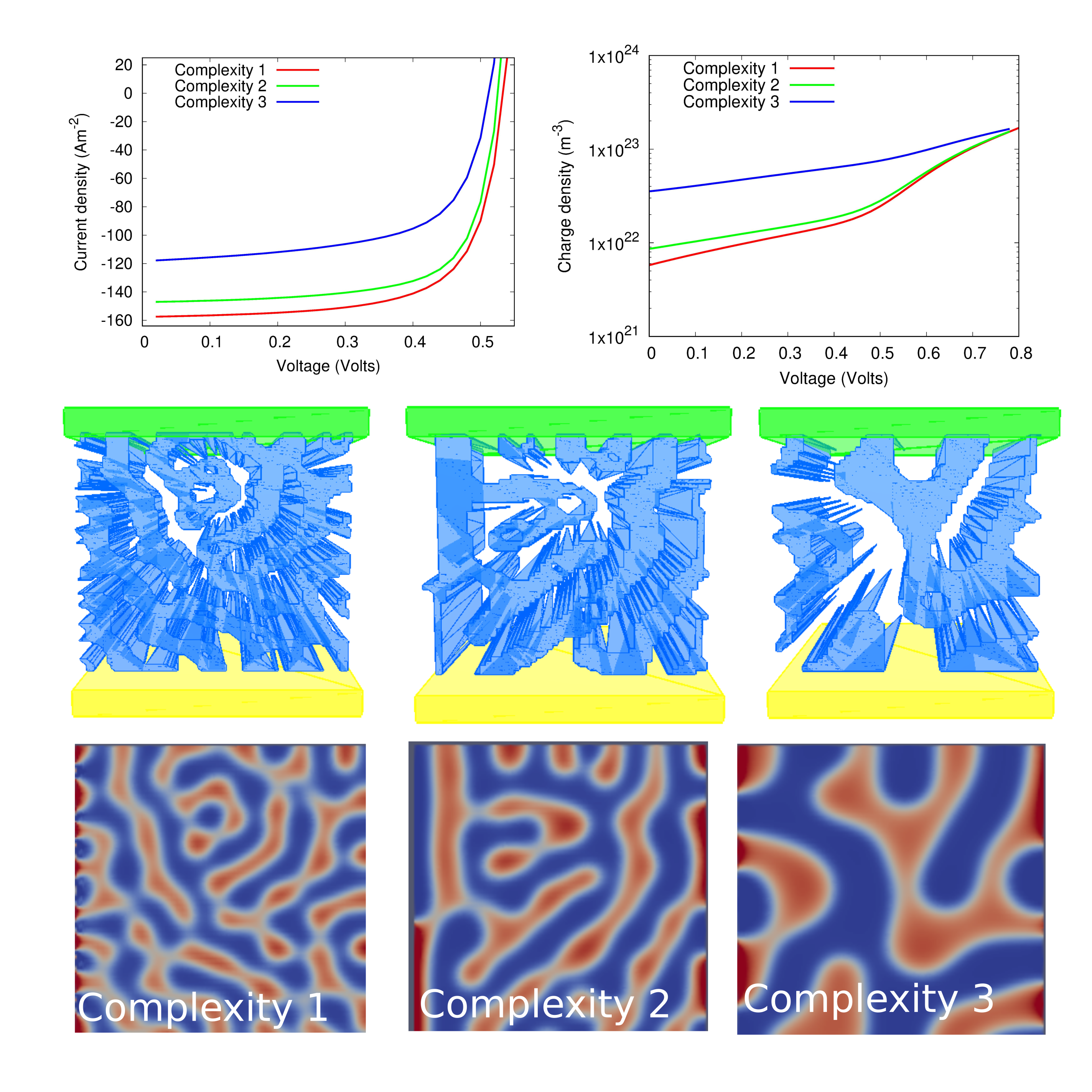}
	\caption{Lower row: Computationally generated morphologies; Middle row: The three morphologies discretized and turned into 3D structures; Top row: Current and carrier density for each structure plotted as a function voltage applied between the top (green) and bottom (yellow) contacts.}
	\label{figure:ofet1}
\end{figure}

Three morphologies were taken from the simulations above with different levels of coarseness/complexity. 
It can be seen from \cref{figure:ofet1} that the more fine grained the morphology the higher the short circuit current ($J_\textup{sc}$) and open circuit voltage ($V_\textup{oc}$) are. 
Open circuit voltage is the voltage produced when the current generated by the cell is zero (intersection of the $x$ axis), and short circuit current is the current produced by the cell when the external voltage is zero (intersection of $y$ axis). 
Higher values of $J_\textup{sc}$ and $V_\textup{oc}$ are generally associated with more efficient devices. 
It can also be seen that the charge density is higher in devices with less fine grained morphology.

\section{Conclusion and Outlook}
\label{section:conclusion}

In this paper we have presented a computationally efficient pipeline for the evaluation of a phase--field model describing the morphology evolution of organic solar cells. 
Our preconditioning strategy enabled the parameter-robust simulation of the discretized equations. 
Additionally, we showcased how the results can be used to further characterize properties of the organic solar cells such as current-voltage characteristics.

In order to be able to compare our results to experimental data, the next step will be to perform a proper non-dimensionalisation using realistic physical parameters such as diffusion constants or Flory--Huggins interaction parameters. 
As a result, large constants might appear in the model, rendering the numerical solution even more challenging. 
Furthermore, to be able to obtain results over longer time scales, \eg as needed to model spin coating experiments, additional homogenisation techniques might need to be employed to obtain coarse-grained models.

\printbibliography

\end{document}